\documentclass[12pt,reqno]{amsart}
\usepackage{amsmath}
\usepackage{amssymb,amsthm,enumerate}
\usepackage{tikz}
\usepackage[pdftex]{hyperref}
\newtheorem{theorem}{Theorem}[section]
\newtheorem{lemma}[theorem]{Lemma}
\newtheorem{proposition}[theorem]{Proposition}
\newtheorem{corollary}[theorem]{Corollary}
\newtheorem{remark}[theorem]{Remark}
\newtheorem{example}[theorem]{Example}
\newtheorem{definition}[theorem]{Definition}
\numberwithin{equation}{section}
\usepackage[margin=2.5cm]{geometry}

\newcommand{\supp}{{\ensuremath{\rm supp }}}
\newcommand{\depth}{{\ensuremath{\rm depth }}}

\begin{document}

\title[Certain Classes of Cohen-Macaulay Multipartite Graphs]{Certain Classes of Cohen-Macaulay Multipartite Graphs}

\author[R. Kumar]{Rajiv Kumar}

\address{Indian Institute of Technology Bombay, India.}

\email{gargrajiv00@gmail.com}
\author[A. Kumar]{Ajay Kumar}

\address{DAV University Jalandhar, India.}

\email{ajaychhabra.msc@gmail.com}

\date{\today}

\maketitle

\begin{abstract}
The Cohen-Macaulay property of a graph arising from a poset has been studied by various authors. In this article, we study the Cohen-Macaulay property of a graph arising from a family of  reflexive and antisymmetric relations on a set. We use this result to find classes of multipartite graphs which are Cohen-Macaulay.  
\end{abstract}

\section{Introduction}

Graphs and simplicial complexes play an important role in combinatorial commutative algebra. In order to see the relationship between commutative algebra and combinatorics, one can associate monomial ideals to graphs or simplicial complexes.
Many authors have studied the connection between the algebraic properties of these ideals and the combinatorial properties of the corresponding combinatorial objects, see \cite[Chapter 9]{HHi}. In this article, our main focus is to study the edge ideal of a graph. A graph is called Cohen-Macaulay if the corresponding edge ideal is Cohen-Macaulay.

The Cohen-Macaulay property of graphs has been well studied for various classes. 
Herzog-Hibi (\cite{HH})
prove that a bipartite graph is Cohen-Macaulay if and only if it is arising from a poset. For a finite poset and $r, s\in \mathbb{N}$, Ene-Herzog-Mohammadi (\cite{EHM}) associated a monomial ideal generated in degree $s$, to the set of all multichains of length $r$ in a poset, and proved that this ideal is Cohen-Macaulay. Note that if $s=2$, then these ideals are edge ideals of some $r$-partite graphs.
Motivated by these results, we associate a monomial ideal to a family of posets, and find a class of Cohen-Macaulay $r$-partite graphs.

Our main tool in this article is the following well known relationship between the Stanley-Reisner ideal and its Alexander dual: The Stanley-Reisner ideal is Cohen-Macaulay if and only if its Alexander dual has a linear resolution. For more details see \cite[Theorem 3]{ER}.

This paper has been organized in the following manner. In Section \ref{Prelim}, we introduce the basic notions which are used throughout the article, more details can be found in \cite{HHi}. In Section \ref{Resolution}, we associate a monomial ideal $H_r(\mathcal{P})$ to a family $\mathcal{P}$ of partial order relations on a finite set. In Lemma \ref{lQLemma}, we prove that monomial ideal $H_r(\mathcal{P})$ has a linear resolution. This forces the Alexander dual of $H_r(\mathcal{P})$ to be Cohen-Macaulay. 

Section \ref{CMGraphs} is devoted to finding the classes of Cohen-Macaulay $r$-partite graphs.
In Theorem \ref{Alexander}, we see that the Alexander dual of $H_r(\mathcal{P})$ is an edge ideal of an $r$-partite graph associated to a family of reflexive and antisymmetric relations on a given set. Using this we find classes of Cohen-Macaulay graphs which are recorded in Theorems \ref{mainTheorem1} and \ref{mainTheorem2}.

\section{Preliminaries}\label{Prelim}
\subsection{Notation}
The following notation is used throughout the article.
\begin{enumerate}[i)]
	\item For $n\in \mathbb{N}$, we denote $[n]=\{1,\dots, n\}$.
	\item By $P_a$, we mean that the set $P$ with partial relation $\leq_a$.
	\item Let $S={\sf k}[X_1,\dots,X_n]$ be a polynomial ring with $\deg(X_i)=1$,  where ${\sf k}$ is a field. Then by $S(-j)$, we mean a graded free $S$-module of rank $1$ with $S(-j)_n=S_{n-j}$.
	\item Let $M$ and $N$ be graded $S$-modules. Then a homomorphism $\phi:M\longrightarrow N$ is called a \emph{graded homomorphism} if $\phi(M_n)\subset N_n$.
	\end{enumerate}
\subsection{Graphs and Edge Ideals} 
\begin{definition}{\rm\hfill{}\\\vspace*{-.5cm}
	\begin{enumerate}[ i)]
		\item A \emph{graph} $G=(V,E)$ is an ordered pair, where $V$ is the set of vertices of $G$ and $E$ is a collection of subsets of $V$ of cardinality $2$.
		\item An element of $E$ is called an \emph{edge} of $G$. For all $i, j\in V$, we say that $i$ is \emph{adjacent} to $j$ if and only if $\{i,j\}\in E$.
		\item For an integer $r\geq 2$, a graph $G$ is called an $r$-partite if there exists a partition of $V=V_1\cup\cdots\cup V_r$ such that for all $1\leq k \leq r$ and $i,j\in V_k$ implies that $i$ is not adjacent to $j$. If $r=2$, we say that $G$ is a \emph{bipartite} graph.  A bipartite graph on vertex set $V=V_1\cup V_2$ is called a \emph{complete bipartite graph} if $i$ and $j$ are adjacent for all $i\in V_1$ and $j\in V_2$.
		\item Let $G$ be a graph on a vertex set $V$ and $W\subset V$. Then a graph $H$ is called a \emph{induced subgraph} of $G$ on $W$ if for $i,j\in W$, $i$ and $j$ are adjacent in $H$ if and only if so in $G$.
		\item Let $S={\sf k}[X_1,\dots, X_n]$ be a polynomial ring over ${\sf k}$ and $G$ be a graph on a vertex set $V=[n]$. Then the monomial ideal $I(G)=\langle X_iX_j: \{i,j\}\in E\rangle$ is called the \emph{edge ideal} of $G$. A graph $G$ is called \emph{Cohen-Macaulay} if $S/I(G)$ is Cohen-Macaulay. 
 	\end{enumerate}
}\end{definition}

\subsection{Simplicial Complexes}\label{SimplicialComplexes}  

\begin{definition}{\rm For fixed $n\in \mathbb{N}$, let $V=[n]$.
		\begin{enumerate}[i)]
			\item A \emph{simplicial complex} on $V$, denoted by $\Delta$ or $\Delta_V$,  is a collection of subsets of $V$ with the following properties:
			\begin{enumerate}[a)]
				\item $\phi \in \Delta$ and $\{i\}\in \Delta$ for all $i\in V$.
				\item If $F\in \Delta$ and $G\subset F$, then $G\in \Delta$.
			\end{enumerate}
			\item An element of $\Delta$ is called a \emph{face} of $\Delta$, and a maximal face with respect to inclusion is called a \emph{facet}.
			\item A subset $F\subset V$ is called a \emph{nonface} of $\Delta$ if $F\notin \Delta$, and it is called a \emph{minimal nonface} if it is minimal with respect to inclusion.
			\item The \emph{Alexander dual} of $\Delta$, denoted by $\Delta^{\vee}$, is defined as 
			$$\Delta^{\vee}=\{V\setminus F: F \text{ is a nonface of } \Delta \}.$$
			\item Let $\Delta$ be a simplicial complex on $[n]$ and $S={\sf k}[X_1,\dots, X_n]$. The \emph{Stanley-Reisner ideal} of $\Delta$ is the squarefree monomial ideal, denote as $I_{\Delta}$, is defined as follows:
			$$I_{\Delta}=\left\langle X_{i_1}\cdots X_{i_r}:\{i_1,\dots, i_r\} \text{ is a minimal nonface of } \Delta\right\rangle.$$
			Further, let $\Delta^{\vee}$ be the Alexander dual of $\Delta$. Then we say that $I_{\Delta^{\vee}}$ is the \emph{Alexander dual} of $I_{\Delta}$, and denote it by $I_{\Delta}^{\vee}$.
		\end{enumerate}
}\end{definition}

\subsection{Free Resolution} Let $S={\sf k}[X_1,\dots, X_n]$ and $I$ be a homogeneous ideal of $S$. Then a \emph{free resolution} of $I$ over $S$ is an exact sequence

$${\sf F}_{\bullet}:\cdots\longrightarrow F_n\overset{\phi_n}\longrightarrow F_{n-1}\longrightarrow\cdots\longrightarrow F_0\longrightarrow I\longrightarrow0$$
such that for each $i\geq 0$, $F_i$ is a graded free $S$-module and $\phi_i$ is a graded homomorphism. 
Further, if $I$ is generated in degree $d$ and $F_i\simeq S(-d-i)^{\beta_i}$ for some $\beta_i\in \mathbb{N}$ and for all $i$, then we say that $I$ has a \emph{linear resolution}.  

\section{Linear Resolution}\label{Resolution}
In this section, we associate a monomial ideal $H_r(\mathcal{P})$  to a family of posets $\mathcal{P}$, and we show that this ideal has a linear resolution.

\begin{definition}{\rm
		Let $(P,\leq)$ be a finite partial ordered set. A subset $I\subset P$ is called a \emph{poset ideal} if for all $p\in I$ and $q\in P$ with $q\leq p$ implies that $q\in I$. 
}\end{definition}
For a given set $P=\{p_1,\dots, p_n\}$ and an integer $r \geq 2$, we consider a family of  partial ordered relations $\mathcal{P}=\{  \leq_a:a \in [r-1] \}$ on $P$. For the sake of simplicity, we denote $(P,\leq_a)$ by $P_a$. Let $J(P_a)$ denotes the set of poset ideals of $P_a$. Now corresponding to a family $\mathcal{P}$, we define a set
\begin{align*}
K_r(\mathcal{P})=
\lbrace {\mathbf{I}}=\left( I_1,I_2,\ldots, I_{r-1} \right): I_a\in J(P_a) ~\forall~ a\in [r-1]\text{ and } I_{r-1}\subset \cdots \subset I_1 \rbrace,
\end{align*}
with a partial order $\prec$ given by ${\mathbf{J}} \prec {\mathbf{I}} \text{ if and only if } J_a \subset I_a~\forall~ a.$ \\
Let $S={\sf k}[{\bf X}_1,{\bf X}_2,\ldots,{\bf X}_r],$ where  $\mathbf{X}_a=\{X_{a,1},\ldots,X_{a,n}\}$  for all $a \in [r]$. For ${\mathbf{I}}\in K_r(\mathcal{P}),$ we associate a squarefree monomial 
$$
u_{{\mathbf{I}}}=u_{I_1}u_{I_2}\cdots u_{I_{r-1}}, \text{ where } u_{I_a}= \left(\prod_{p_i\in I_a}X_{a,i}\prod_{p_i \in P_a \setminus I_a}X_{{a+1},i}   \right) \text{ for all } 1\leq a \leq r-1.$$
Let $H_r(\mathcal{P})=\left\langle \{ u_{{\mathbf{I}}} \}_{\mathbf{I}\in K_r(\mathcal{P})} \right\rangle$ be the squarefree monomial ideal of the polynomial ring $S$, generated by monomials $u_{{\mathbf{I}}}$, where $\mathbf{I}\in K_r(\mathcal{P}) $. 

\begin{example}\label{posetideal}{\rm
		Let $P=\{p_1,p_2,p_3\}$ and $r=3$. Suppose we have the following partial order relations on $P$:
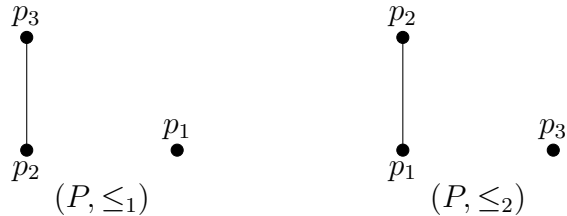
		\[
		\begin{tikzpicture}[baseline=(current bounding box.north), level/.style={sibling distance=50mm}]	
		\tikzstyle{edge} = [draw,thick,-]	
		\draw[-] (0,0) -- (0,1.5);
		\draw [fill] (0,0) circle [radius=0.08];
		\draw [fill] (0,1.5) circle [radius=0.08];
		\draw [fill] (2,0) circle [radius=0.08];		
		\node [above] at (0,1.5) { $p_3$};
		\node [below] at (0,0) { $p_2$};
		\node [above] at (2,0) { $p_1$};
		\node[below] at (1,-.3){$(P,\leq_1)$};
		
		\draw[-] (5,0) -- (5,1.5);
		\draw [fill] (5,0) circle [radius=0.08];
		\draw [fill] (5,1.5) circle [radius=0.08];	
		\draw [fill] (7,0) circle [radius=0.08];		
		\node [above] at (5,1.5) { $p_2$};
		\node [below] at (5,0) { $p_1$};
		\node [above] at (7,0) { $p_3$};
		\node[below] at (6,-.3){$(P,\leq_2)$};
		
		\end{tikzpicture}
		\]
Then the collection of poset ideals of $P_1$ is $J(P_1)=\{\emptyset, \{ p_1\}, \{ p_2\},\{p_1,p_2\}, \{p_2,p_3 \},P \}$ and, of $P_2$  is $J(P_2)=\{\emptyset, \{ p_1\}, \{ p_3\},\{p_1,p_2\}, \{p_1,p_3 \},P \}$. Also, We see that 

\[
K_3(\mathcal{P})=\left\{
\begin{array}{c}
	(\phi, \phi),(\{ p_1\}, \phi ), (\{ p_1\}, \{ p_1\} ), (\{ p_2\}, \phi ), (\{ p_1, p_2\}, \phi ), \\
	 (\{ p_1, p_2\}, \{ p_1\}), 
	(\{ p_1, p_2\}, \{p_1,p_2\}), 
	(\{ p_2, p_3\}, \phi ), 
	(\{ p_2, p_3\}, \{ p_3\} ), \\
 (P, \phi ), 
	(P,\{ p_1\}), 
	(P,\{ p_3\}), (P,\{p_1,p_2\}), (P,\{p_1,p_3\}), (P,P) \}
\end{array} \right\},
\]
and hence $H_3(\mathcal{P})$ is generated by the following set of squarefree monomials
\[
\left\{
\begin{array}{c}
 X_{2,1}X_{2,2}X_{2,3}X_{3,1}X_{3,2}X_{3,3}, X_{1,1}X_{2,2}X_{2,3}X_{3,1}X_{3,2}X_{3,3},X_{1,1}X_{2,1}X_{2,2}X_{2,3}X_{3,2}X_{3,3}, \\
X_{1,2}X_{2,1}X_{2,3}X_{3,1}X_{3,2},X_{3,3},
X_{1,1}X_{1,2}X_{2,3}X_{3,1}X_{3,2}X_{3,3}, X_{1,1}X_{1,2}X_{2,1}X_{2,3}X_{3,2}X_{3,3}, \\
X_{1,1}X_{1,2}X_{2,1}X_{2,2}X_{2,3}X_{3,3}, X_{1,2}X_{1,3}X_{2,1}X_{3,1}X_{3,2}X_{3,3}, 
,X_{1,2}X_{1,3}X_{2,1}X_{2,3}X_{3,1}X_{3,2} \\
X_{1,1}X_{1,2}X_{1,3}X_{3,1}X_{3,2}X_{3,3},  
X_{1,1}X_{1,2}X_{1,3}X_{2,1}X_{3,2}X_{3,3}, X_{1,1}X_{1,2}X_{1,3}X_{2,3}X_{3,1}X_{3,2}, \\
X_{1,1}X_{1,2}X_{1,3}X_{2,1}X_{2,2}X_{3,3},
X_{1,1}X_{1,2}X_{1,3}X_{2,1}X_{2,3}X_{3,2}, X_{1,1}X_{1,2}X_{1,3}X_{2,1}X_{2,2}X_{2,3}
\end{array}\right\}
\]
in the polynomial ring ${\sf k}[X_{a,i}, 1 \leq a,i \leq 3]$.
}
\end{example}

Note that For $r=2$, Herzog-Hibi (\cite{HH}) proved that the ideal $H_2(\mathcal{P})$ has a linear resolution. In fact, if all partial order relations $\leq_a$ are same for $a\in [r-1]$, then Ene-Herzog-Mohammadi (\cite{EHM}) studied the ideal $H_r(\mathcal{P})$, and proved that it has a linear resolution. More generally, in Proposition \ref{linear}, we prove that $H_r(\mathcal{P}) $ has a linear resolution.

Let $G(H_r(\mathcal{P}))$ denotes the minimal set of monomial generators of the monomial ideal $H_r(\mathcal{P}).$ Define a partial order on $G(H_r(\mathcal{P}))$ by $u_{\mathbf{J}}\prec u_{\mathbf{I}}$ if $\mathbf{J}\prec \mathbf{I}.$ We fix a total order $\prec'$ on $G(H_r(\mathcal{P}))$, which extends the partial order $\prec $. For more details see \cite[Theorem 1.1]{ES}. 

 In order to prove that $H_r(\mathcal{P})$ has a linear resolution, we use the following remark.
\begin{remark}\label{lQLemma}{\rm
Let $I$ be a monomial ideal with monomial generators $u_1,\dots,u_m$ of $I$. Suppose for all $j<i$, there exists an integer $k<i$ and an integer $l$ such that 

$$\dfrac{u_k}{\gcd(u_k,u_i)}=X_l \quad \text{and }\quad X_l \quad \text{divides} \quad \dfrac{u_j}{\gcd(u_j,u_i)}.$$ 

Then, by \cite[Theorem 8.2.1 and Lemma 8.2.3]{HHi}  $I$ has a linear resolution.
}\end{remark}

\begin{proposition}\label{linear}
The squarefree monomial ideal $H_r(\mathcal{P})$ has a linear resolution.
\end{proposition}
\begin{proof}
 Let ${\mathbf{J}}\prec'{\mathbf{I}}$ with ${\mathbf{J}}=\left(J_1,J_2,\ldots,J_{r-1} \right),\ {\mathbf{I}}=\left(I_1,I_2,\ldots,I_{r-1}  \right).$ For $a\in[r-1]$, define $J'_a= J_a\cap I_a$ and ${\bf J}'=(J_1',\dots, J_{r-1}' )$.  Since $J_a'\subset I_a$ for $a\in[r-1]$, we have ${\bf J'}\prec{\bf I}$.

 Then take $a=\max \{q:J_q'\subsetneq I_q\}.$  Since $J_a'\subsetneq I_a$ are poset ideals of $P_a$, there exists a $p_i\in I_a\setminus J_a'$ such that $p_i$ is a maximal element in $I_a$. This forces that  $\delta_a=I_a\setminus \{ p_i\}$ is a poset ideal of $P_a$. This gives us $u_{\delta_a}=\dfrac{u_{I_a}X_{a+1,i}}{X_{a,i}}$ which implies that $X_{a+1,i}=\dfrac{u_{\delta_a}}{\gcd(u_{\delta_a},u_{I_a})}.$

Our claim is 
${\mathbf{K}}=\left(I_1,I_2,\ldots,I_{a-1},\delta_a,I_{a+1},\ldots,I_{r-1}  \right)\in K_r(\mathcal{P})$. Since $\delta_a\subsetneq I_a$, to prove the claim it is enough to prove that $I_{a+1}\subset \delta_a$. The fact that $J_{a+1}'\subset J_a'$ and $p_i\notin J_a'$ implies that $p_i\notin J_{a+1}',$ and by choice of $a$, we know that $J_{a+1}'=I_{a+1}$. The claim follows from the fact that $\delta_a=I_a\setminus p_i$.

Since we know that $p_i\in I_a$ and $p_i\notin J_{a+1}'=I_{a+1},$ we get $X_{a+1,i}$ does not divide $u_{\mathbf{I}}.$ Now, $p_i \notin J_a'$ and $p_i \in I_a$ and hence $p_i \notin J_a$. This forces that  $X_{a+1,i}$ divides $u_{J_a}$, and hence we get  
$X_{a+1,i}$ divides $\dfrac{u_{\mathbf{J}}}{\gcd(u_{\mathbf{J}},u_{\mathbf{I}})}.$

Finally, from the definition of ${\bf K}$, we get $\dfrac{u_{\bf K}}{\gcd(u_{\bf I},u_{\bf K})}=\dfrac{u_{\delta_a}}{\gcd(u_{I_a},u_{\delta_a})}=X_{a+1,i}$, and hence
the proof follows from Remark \ref{lQLemma}.
\end{proof}

\section{Cohen Macaulay Multipartite Graphs}\label{CMGraphs}
For $a\in [r-1]$ and $P=\{ p_1,p_2, \ldots, p_{n}\}$, let $\leq_a$ be a partial order relation on $P$ such that if $p_i \leq_a p_j$, then we have $i \leq j.$
In this case, we prove that the Alexander dual of $H_r(\mathcal{P})$ is the edge ideal of some $r$-partite graph. Using this we identify two classes of Cohen-Macaulay graphs.
  In order to find $H_r(\mathcal{P})^{\vee}$, we define the following relation on $P$:
For $1\leq a\leq b\leq r-1$, we define a relation $\leq_{[a,b]}$ on $P$ as follows:
$p_{i}\leq_{[a,b]} p_{j}$~ if there exists a non-decreasing sequence $(t_1,t_2,\ldots,t_k)$ such that  $$p_{i} \leq_{a} p_{t_1} \leq_{a+1}\cdots\leq_{b-1}p_{t_k}\leq_{b}p_{j}.$$ 

\begin{example}\label{posetExample}{\rm
		Let $\leq_1$ and $\leq_2$ be partial order relations on a set $P=\{p_1,p_2,p_3\}$ as defined in the Example \ref{posetideal}.
		Then the relation $\leq_{[1,2]}$ on $P$ is shown as in the following diagram. 
		
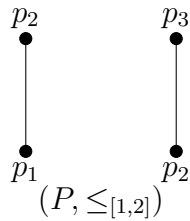
\[
		\begin{tikzpicture}[baseline=(current bounding box.north), level/.style={sibling distance=50mm}]
		\tikzstyle{edge} = [draw,thick,-]		
		\draw[-] (0,0) -- (0,1.5);
		\draw[-] (2,0) -- (2,1.5);
		\draw [fill] (0,0) circle [radius=0.08];
		\draw [fill] (0,1.5) circle [radius=0.08];		
		\draw [fill] (2,1.5) circle [radius=0.08];		
		\draw [fill] (2,0) circle [radius=0.08];		
		\node [above] at (0,1.5) { $p_2$};
		\node [below] at (0,0) { $p_1$};
		\node [above] at (2,1.5) {$p_3$};	
		\node [below] at (2,0) { $p_2$};
		\node[below] at (1,-.3){$(P,\leq_{[1,2]})$};
		\end{tikzpicture}
		\]		
}\end{example}

For a simplicity, we denote $P_{[a,b]}=(P,\leq_{[a,b]})$. Since $\leq_a$ is a reflexive relation on $P$ for all $a\in [r-1]$, by definition of $\leq_{[a,b]}$ it follows that so is $\leq_{[a,b]}$. Further, the relation $\leq_{[a,b]}$ on $P$ is antisymmetric follows from the fact that if $p_i \leq_a p_j$, then we have $i \leq j.$ In Example \ref{posetExample} observe that $\leq_{[1,2]}$ is not a transitive relation on $P$, and hence 
 $P_{[a,b]}$ need not be a poset.

 \begin{definition}{\rm\hfill{}\\\vspace*{-.5cm}
 		\begin{enumerate}[i)]
 			\item Let $u$ be a monomial in $S$. Then \emph{support} of $u$, denoted by $\supp(u)$, is defined as $$\supp(u)=\{X_{a,i}: X_{a,i}\text{ divides }  u\}.$$ 
 			\item Let $V=\bigcup\limits_{a=1}^rV_a$, where $V_a=\{X_{a,1},\dots,X_{a,n}\}$. For $F\subset V$,  define $F_{{\bf X}_a}=\{X_{1,i}:X_{a,i}\in F\}$ for $a\leq r$.
 			\item For $F\subset V$, we set $ \gamma^F_r =\phi.$ For $2\leq a\leq r$, we define a poset ideal, denoted as $\gamma^F_{a-1}$, generated by $\{p_i: X_{1,i} \in F_{{\bf X}_{a}} \}\cup \gamma^F_{a}$  in $P_{a-1}$. Note that $\gamma_{a+1}^F\subset \gamma_{a}^F$ for all $a\in [r-1]$, so we denote $\boldsymbol{\gamma}^F=(\gamma_1^F,\ldots,\gamma_{r-1}^F)\in K_r(\mathcal{P})$.
 		\end{enumerate} 
 }\end{definition}
 \begin{lemma}\label{remark}
 	Let $p_i\in \gamma^F_a$ for some $a\in[r-1]$. Then there exists $p_j$ such that $X_{1,j}\in F_{{\bf X}_b}$ for some $a+1\leq b\leq r$ with $p_i\leq_{[a,b-1]}p_j$. 
 \end{lemma}
 \begin{proof}
 Since $p_i\in \gamma^F_a$, there exists a maximal element $p_j \in \gamma^F_a$ with $p_i \leq _a p_j$. Now by definition of $\gamma^F_a$, it follows that either $X_{1,j} \in F_{\mathbf{X}_{a+1}}$ or $p_j \in \gamma^F_{a+1}$.  If $X_{1,j} \in F_{\mathbf{X}_{a+1}}$, then we are through. Otherwise there exists a maximal element $p_k \in \gamma^F_{a+1} $ with $p_j \leq _{a+1} p_k$. Again we have either $X_{1,k} \in F_{\mathbf{X}_{a+2}}$ or $p_k \in \gamma^F_{a+2}$. Thus, we repeat this procedure till we get the desired result.
 \end{proof}
 
\begin{theorem}\label{Alexander}
The monomial ideal $H_r(\mathcal{P})^{\vee}$ is minimally generated by the squarefree monomials of type $X_{s,i}X_{t,j}$, $1\leq s<t\leq r$ if and only if $p_i\leq_{[s,t-1]} p_j.$
\end{theorem}

\begin{proof}
Let $V=\bigcup\limits_{a=1}^rV_a$ be a vertex set and $\Delta_P$ be a simplicial complex on $V$ such that $I_{\Delta_P}=H_r(\mathcal{P}).$ Set $w$ to be the product of all variables. 
By the definition of $\Delta_P^{\vee}$, facets of $\Delta_P^{\vee}$ are given by $\supp(w/u_{\mathbf{I}})$, where ${\mathbf{I}} \in K_r(\mathcal{P})$. Now, by definition of $u_{\bf I}$, it follows that $X_{a,i}$ divides $u_{\mathbf{I}}$ if and only if $p_i\notin I_{a-1}$ or $p_i\in I_a$. 
Also, we know that $$\left( \supp({w/u_{\mathbf{I}}}) \right)_{\mathbf{X}_a}=\{X_{1,i}: X_{a,i} \text{ does not divide } u_{\bf I}\},$$ and hence, we get  $\left( \supp({w/u_{\mathbf{I}}}) \right)_{{\bf X}_a}=\{ X_{1,i}:p_i \in I_{a-1}\setminus I_{a}\}$. 
 
Let $F \subset V$ be a face of $\Delta_P^{\vee}$. Since facets of $\Delta_P^{\vee}$ are given by $\supp(w/u_{\mathbf{I}})$ for some ${\mathbf{I}} \in K_r(\mathcal{P})$, there exists 
${\mathbf{I}} \in K_r(\mathcal{P})$ such that $F\subset \supp({w/u_{\mathbf{I}}})$. In particular,  for $a\in [r]$, we have $F_{\mathbf{X}_a} \subset \left( \supp({w/u_{\mathbf{I}}}) \right)_{\mathbf{X}_a}$.

First, assume that $p_i \leq_{[a,b-1]} p_j$, where $1\leq a<b\leq r$. Then our claim is that the set $F=\{X_{a,i},X_{b,j}\}$ is a minimal non-face of $\Delta_P^{\vee}$. Note that since $|F|=2$, it is enough to prove that $F$ is a non-face of $\Delta_P^{\vee}$.

Let $\mathbf{I}\in K_r(\mathcal{P}).$ If $p_j \notin I_{b-1} \setminus I_b,$ then $X_{b,j}$ divides $u_{\bf I}$, and hence  $X_{b,j}\notin \supp({w/u_{\bf I}})$. This forces that $F_{{\bf X}_b} \not\subset \left( \supp({w/u_{\mathbf{I}}}) \right)_{{\bf X}_b}.$ Otherwise, we claim that $F_{{\bf X}_a} \not\subset \left( \supp({w/u_{\mathbf{I}}}) \right)_{{\bf X}_a},$ and hence by the claim $F$ is a non-face. In order to prove the claim, it is enough to prove that $p_i\in I_a $.

\emph{Proof of claim:}
Since $p_i \leq_{[a,b-1]} p_j,$ there exist a non-decreasing sequence $(t_1,t_2,\ldots,t_k)$ such that $p_{i} \leq_{a} p_{t_1} \leq_{a+1}\cdots\leq_{b-2}p_{t_k}\leq_{b-1}p_{j}$. The fact $p_{t_k}\leq_{b-1}p_{j}$ and 
$I_{b-1}$ is a poset ideal in $P_{b-1}$ implies that $p_{t_k} \in I_{b-1}$. Using $I_{b-1}\subset I_{b-2}$, we get $p_{t_k} \in I_{b-2}$. Now, again repeat the process, we get $p_i \in I_a.$ This proves the claim.

Conversely, let $F$ be a minimal nonface of 
$\Delta_P^{\vee}$. 
This gives us the following:
\begin{equation}\label{non}
\text{for~ any}~ \mathbf{I} \in K_r(P)~ \exists~ a \in [r]~ \text{such ~that}~ F_{\mathbf{X}_a} \not\subset \left( \supp({w/u_{\mathbf{I}}}) \right)_{\mathbf{X}_a}.
\end{equation}

\emph{Case 1}: Suppose $F_{{\bf X}_{a}} \not\subset (\supp({w/u_{\boldsymbol{\gamma}^F}}))_{{\bf X}_a}$  for some $ 2\leq a \leq r$. Since we know that $(\supp({w/u_{\boldsymbol{\gamma}^F}}))_{{\bf X}_a}=\{ X_{1,i}:p_i \in \gamma_{a-1}^F\setminus \gamma_{a}^F\}$, there exists some $X_{1,i} \in F_{{\bf X}_{a}}$ such that either $p_i \notin \gamma^F_{a-1}$ or $p_i \in \gamma^F_{a}.$ Also, note that $X_{1,i} \in F_{{\bf X}_{a}}$, by definition of $\gamma^F_{a-1},$ we see that $p_i\in\gamma^F_{a-1}$, and hence $p_i\in\gamma^F_{a}$. By Lemma \ref{remark}, there exists a $p_j$ such that $p_i\leq_{[a,b-1]}p_j$ with $X_{1,j} \in F_{{\bf X}_b}$ for some $a+1 \leq b \leq r.$
Since $\{ X_{{a},i}, X_{b,j} \}$ is a minimal nonface of $\Delta_P^{\vee}$. By assumption $F$ is a minimal nonface, and hence we get $F=\{ X_{a,i}, X_{b,j} \}$ with $p_i \leq_{[a,b-1]} p_j.$
\\
\emph{Case 2}: $F_{{\bf X}_{a}} \subset (\supp({w/u_{\boldsymbol{\gamma}^F}}))_{{\bf X}_a}$ for $2 \leq a \leq r$.
Since $\boldsymbol{\gamma}^F\in K_r(\mathcal{P})$, by Equation \eqref{non}, we get $F_{{\bf X}_1}\not\subset (\supp({w/u_{\boldsymbol{\gamma}^F}}))_{{\bf X}_1}$. Thus, there exists $p_i \in \gamma^F_1$ such that $X_{1,i} \in F_{{\bf X}_1}.$

Since $p_i\notin\gamma^F_r=\phi$, we can choose $2\leq b \leq r$ such that $b=\min\{j:p_i\notin \gamma^F_j\}$. Then, we see that $p_i \in \gamma^F_{b-1} \setminus \gamma^F_{b}$. 
Let $p_j\in \gamma^F_{b-1} \setminus \gamma^F_{b}$ such that $p_i \leq_{b-1} p_j.$ Since $p_j\in \gamma^F_{b-1}\setminus\gamma^F_{b}$, by Lemma \ref{remark}, there exists some $b\leq c\leq r$ such that $p_j\leq_{[b,c-1]}p_k$ with $X_{c,k}\in F$. Note that $p_i\leq_{[1,c-1]}p_k$ and $\{X_{1,i}, X_{c,k}\}\subset F$. By assumption $F$ is a minimal nonface of $\Delta^{\vee}_P$ and we know $\{X_{1,i}, X_{c,k}\}\subset F$ is a nonface, and hence $F=\{X_{1,i}, X_{c,k}\}$ with $p_i\leq_{[1,c-1]}p_k$.
 \end{proof}

Let $\mathcal{F}=\{\leq_{[a,b]}:1\leq a\leq b\leq r-1 \}$ be a family of reflexive and antisymmetric relations on a given set $P=\{p_1,p_2,\dots,p_n \}$. Now corresponding to $\mathcal{F}$, we associate an $r$-partite graph on a vertex set $V=\bigcup\limits_{a=1}^rV_a$ such that $ X_{a, i}$ is adjacent to $ X_{b,j}$  if and only if $p_i \leq_{[a,b-1]} p_j$.

\begin{example}{\rm
	Let $P=\{p_1,p_2,p_3\}$ and $\mathcal{F}=\{\leq_1,\leq_2,\leq_{[1,2]}\}$, where $\leq_1, \leq_2,\leq_{[1,2]}$ are given as in Examples \ref{posetideal} and \ref{posetExample}. We associate a following graph on vertices set $V=\{X_{a,i}: a,i\in [3]\}$:
	
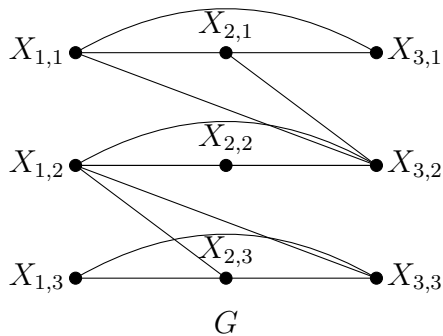
	\[
	\begin{tikzpicture}[baseline=(current bounding box.north), level/.style={sibling distance=50mm}]
	\tikzstyle{edge} = [draw,thick,-]		
	\draw[-] (0,0) -- (2,0);
	\draw[-] (0,1.5) -- (2,0);
	\draw[-] (0,0) to[out=30, in=150] (4,0);
	\draw[-] (0,1.5) -- (2,1.5);
	\draw[-] (0,1.5) -- (4,1.5);
	\draw[-] (0,3) -- (4,1.5);
	\draw[-] (0,3) to[out=30, in=150] (4,3);
	\draw[-] (0,3) -- (2,3);
	\draw[-] (2,3) -- (4,3);
	\draw[-] (2,3) -- (4,1.5);
	\draw[-] (2,1.5) -- (4,1.5);
	\draw[-] (2,0) -- (4,0);
	\draw[-] (0,1.5) -- (4,0);
	\draw[-](0,1.5) to[out=30, in=150] (4,1.5);
			
	\node [left] at (0,0) { $X_{1,3}$};
	\node [left] at (0,1.5) { $X_{1,2}$};
	\node [left] at (0,3) { $X_{1,1}$};
	
	\node [above] at (2,0) { $X_{2,3}$};
	\node [above] at (2,1.5) { $X_{2,2}$};
	\node [above] at (2,3) { $X_{2,1}$};
	
	\node [right] at (4,0) { $X_{3,3}$};
	\node [right] at (4,1.5) { $X_{3,2}$};
	\node [right] at (4,3) { $X_{3,1}$};

	\draw [fill] (0,0) circle [radius=0.08];
	\draw [fill] (0,1.5) circle [radius=0.08];
	\draw [fill] (0,3) circle [radius=0.08];
	
	\draw [fill] (2,0) circle [radius=0.08];
	\draw [fill] (2,1.5) circle [radius=0.08];
	\draw [fill] (2,3) circle [radius=0.08];
	
	\draw [fill] (4,0) circle [radius=0.08];
	\draw [fill] (4,1.5) circle [radius=0.08];
	\draw [fill] (4,3) circle [radius=0.08];
	\node[below] at (2,-.3){$G$};
	
	\end{tikzpicture}
	\]
	Note that by the following corollary $G$ is Cohen-Macaulay.
}\end{example}
 \begin{corollary}\label{par}
 Assume that the above family $\mathcal{F}$ have the following properties:
 \begin{enumerate}[\rm i)]
 	\item For $a\in [r-1]$, $P_{[a,a]}=P_a$ is a poset.
 	\item For $1\leq a \leq b\leq r-1$, if $p_i \leq_{[a,b]} p_j$, then $i\leq j$.
 	\item For $1\leq a \leq b\leq r-1$, if $p_i \leq_{[a,b]} p_j$, then there exists a non-decreasing sequence $(t_1,t_2,\ldots,t_k)$ such that  $p_{i} \leq_{a} p_{t_1} \leq_{a+1}\cdots\leq_{b-1}p_{t_k}\leq_{b}p_{j}.$
 \end{enumerate}	
 	
 Then an $r$-partite graph associated to $\mathcal{F}$ is Cohen-Macaulay. 
 \end{corollary}
 \begin{proof}
 We notice that the edge ideal of an $r$-partite graph associated to a family $\mathcal{F}$ is equals to $H_r(\mathcal{P})^{\vee}.$ 
 Since, by Proposition \ref{linear}, the ideal $H_r(\mathcal{P})$ has a linear resolution, the result follows.
 \end{proof}
 
 \begin{theorem}\label{mainTheorem1}
 Let $G$ be an $r$-partite graph with partitions $V_a\ =\{  X_{a,1},\dots, X_{a,n} \},$ for all $a \in [r]$  satisfying the following conditions:
 \begin{enumerate}[\rm i)]
 \item $\{X_{a,i}, X_{b,i}\}$ is an edge for all $i \in [n]$ and $1 \leq a<b \leq r;$
 \item if $\{X_{a,i}, X_{b,j} \}$ is an edge with $1\leq a<b \leq r,$ then $i \leq j;$
\item $\{X_{a,i}, X_{b,j} \}$ is an edge if and only if there exists a non-decreasing sequence $(t_1,t_2,\ldots,t_k)$ such that $\{X_{a,i}, X_{a+1,t_1} \},\dots, \{X_{b-1,t_k}, X_{b,j} \}$ are edges of $G$ for $1\leq a\leq b\leq r$ and $i,j\in [n]$.
 \item if  $\{X_{a,i}, X_{a+1,j} \}$ and $\{X_{a,j}, X_{a+1,k} \}$ are edges for $a\in [r-1]$,   then $\{X_{a,i}, X_{a+1,k} \}$ is an edge.
 \end{enumerate}
 Then $G$ is Cohen-Macaulay.
 \end{theorem}
\begin{proof}
Let $P=\{p_1,p_2,\ldots,p_n \}.$ Then for each $1\leq a<b\leq r,$ we define a relation $\leq_{[a,b-1]}$ on 
$P$ by $p_i \leq_{[a,b-1]} p_j$ if and only if  
$\{X_{a,i},X_{b,j}\}$ is an edge of $G$. Observe that for all $1\leq a< b\leq r$, $P_{[a,b-1]}$ satisfies the hypothesis of Corollary \ref{par}, and hence $G$ is Cohen-Macaulay.
\end{proof}

\begin{definition}{\rm
Let $G$ be a bipartite graph on the vertex set $V$ with partition of $V=V_1\cup V_2$, where $V_1=\{X_{1,1},\dots,X_{1,m}\}$ and $V_2=\{X_{2,1},\ldots,X_{2,n}\}$. Then we say that $G$ satisfies \emph{Herzog-Hibi conditions} if it satisfy the following:
\begin{enumerate}[\rm i)]
\item $m=n$.
\item For all $i$, we have $\{X_{1,i},X_{2,i}\}\in E$.
\item If $\{X_{1,i},X_{2,j}\}\in E$, then $i\leq j$.
\item If $\{X_{1,i},X_{2,j}\}\in E$ and $\{X_{1,j},X_{2,k}\}\in E$, then $\{X_{1,i},X_{2,k}\}\in E$.
\end{enumerate}
Further, if for $i\leq j$, we have $\{X_{1,i},X_{2,j}\}\in E$, then we say $G$ is a \emph{Herzog-Hibi complete} graph.
}\end{definition}
\begin{remark}\label{biparRem}{\rm Let $V_1=\{X_{1,1},\dots,X_{1,m}\}$ and $V_2=\{X_{2,1},\ldots,X_{2,n}\}$, and $G$ be a bipartite graph on a vertex set $V=V_1\cup V_2$. If $G$ satisfy Herzog-Hibi conditions, then, by Theorem \ref{mainTheorem1}, we know that $G$ is Cohen-Macaulay. In fact, Herzog-Hibi (\cite{HH}, 2005) prove that $G$ is Cohen-Macaulay if and only if $G$ satisfy Herzog-Hibi conditions. In this case, further we have the following:
	\begin{enumerate}[i)]
		\item Let $G'$ be induced subgraph on vertex set $V\setminus\{X_{1,1},X_{2,1}\}$. Then note that $G'$ satisfy Herzog-Hibi conditions, and hence $G'$ is Cohen-Macaulay. 
		\item Let $S={\sf k}[X_{1,1},\dots, X_{1,n},X_{2,1},\dots, X_{2,n}]$. Since $X_{2,1}$ is adjacent to only $X_{1,1}$,  we get $\langle I(G), X_{1,1}\rangle=\langle I(G'), X_{1,1}\rangle$ and $\dim\left(\dfrac{S}{\langle I(G), X_{1,1}\rangle}\right)=\dim\left(\dfrac{S}{I(G)}\right)=n$, and hence by (i), we know that $\depth\left(\dfrac{S}{\langle I(G), X_{1,1}\rangle}\right)=n$.
		\item Consider the sequence $0\longrightarrow\dfrac{S(-1)}{ I(G):X_{1,1}}\overset{X_{1,1}}\longrightarrow\dfrac{S}{ I(G)}\longrightarrow \dfrac{S}{\langle I(G),X_{1,1}\rangle}\longrightarrow 0.$ This forces that $\dim\left(\dfrac{S(-1)}{I(G):X_{1,1}}\right)\leq n$ and $\depth\left(\dfrac{S(-1)}{ I(G):X_{1,1}}\right)\geq n$, and hence $\dfrac{S(-1)}{ I(G):X_{1,1}}$ is Cohen-Macaulay of dimension $n$.
	\end{enumerate}
}\end{remark}
\begin{theorem}\label{mainTheorem2}
Let $G$ be an $r$-partite graph with partition $V_1,V_2,\dots, V_r$  which satisfy the following:
\begin{enumerate}[\rm i)]
	\item $|V_a|= n$ for all $1\leq a\leq  r$
	\item for all $1\leq a,b\leq r-1$, the induced graph on the vertices $V_a\cup V_b$ is a complete bipartite graph,
	\item for all $1\leq a\leq r-1$, the induced graph on vertices $V_a\cup V_r$ satisfies the Herzog-Hibi conditions.
\end{enumerate}
Then $G$ is Cohen-Macaulay.
\end{theorem}
\begin{remark}\label{mulparRem}{\rm Let $G$ be a graph as in Theorem \ref{mainTheorem2} and $S={\sf k}[\mathbf{X}_1,\mathbf{X}_2,\dots,\mathbf{X}_n]$. Suppose $V_a=\{X_{a,1},\dots,X_{a,n}\}$ for all $a\in [r]$.
\begin{enumerate}[i)]
	\item If $n=1$, then $G$ is a complete graph on $r$ vertices, and hence $G$ is Cohen-Macaulay.
	\item For $1\leq a\leq r$, let $V'_a=V\setminus\{X_{a,1}\}$, and let $G'$ be the induced subgraph of $G$ on the vertex set $V'=V_1'\cup\cdots\cup V_{r}'.$
	Then $G'$ also satisfies the hypothesis of Theorem \ref{mainTheorem2} with  $|V_a'|=n-1$ for all $a$.
	\item Let $G_a$ be the induced subgraph on vertex set $V_a\cup V_{r}$ for all $a\in [r-1]$. Since, for all $a$, $G_a$ is bipartite graph which satisfies Herzog-Hibi conditions, by Remark \ref{biparRem}(i), $G_a$ is Cohen-Macaulay, and hence by \ref{biparRem}(iii), we get $\dfrac{S}{I(G_a): X_{a,1}} $ is Cohen-Macaulay.
	\item For $a\in[r]$ and $i\in[n]$, define st$(X_{a,i})=\{X_{b,j}\in V: X_{b,j} \text{ adjacent to } X_{a,i}\}$. Since the induced graph on vertices $V_a\cup V_{b}$ is a complete bipartite graph for all $a,b\leq r-1$, we get $V_1\cup\cdots\cup V_{a-1}\cup V_{a+1}\cup\cdots\cup V_{r-1}\subset \text{st}(X_{a,1})$ for all $a\in[r-1]$.
	\item By (iv), we get $$I(G):X_{a,1}=\langle V_1\cup\cdots\cup V_{a-1}\cup V_{a+1}\cup\cdots\cup V_{r-1}, I(G_a):X_{a,1}\rangle.$$ Since, by Remark \ref{biparRem}(iii), $\dfrac{S}{I(G_a):X_{a,1}}$ is Cohen-Macaulay,  
	
	By Remark \ref{biparRem}(iii), we know that $\dfrac{S}{I(G):X_{a,1}}$ is Cohen-Macaulay, and hence so is $\dfrac{S}{\langle I(G), X_{1,1},\dots,X_{a-1,1}\rangle:X_{a,1} }$.
	 Note that $\dfrac{S}{\langle I(G), X_{1,1},\dots,X_{a-1,1}\rangle:X_{a,1} }\simeq \dfrac{{\sf k}[V_a,V_r]}{I(G_a): X_{a,1}}$, and hence Remark \ref{biparRem}(iii) forces that $\dim\left(\dfrac{S}{\langle I(G), X_{1,1},\dots,X_{a-1,1}\rangle:X_{a,1} }\right)=n$.
\end{enumerate}
}\end{remark}
\begin{proof}[Proof of Theorem \ref{mainTheorem2}] 
In order to prove the result, we use the induction on $n$. If $n=1$, by Remark \ref{mulparRem}(i), we know $G$ is Cohen-Macaulay.
If $n>1$, then take $V_a=\{X_{a,1},X_{a,2},\dots,X_{a,n}\}$. Take the following short exact sequence 
{\footnotesize $$	0\longrightarrow \dfrac{S(-1)}{\langle I(G),X_{1,1},\ldots,X_{r-2,1} \rangle:X_{r-1,1}}\overset{X_{r-1,1}}\longrightarrow\dfrac{S}{\langle I(G),X_{1,1},\ldots,X_{r-2,1} \rangle}\longrightarrow \dfrac{S}{\langle I(G),X_{1,1},\ldots,X_{r-1,1} \rangle}\longrightarrow 0.$$}

For $1\leq a\leq r$, let $V'_a=V\setminus\{X_{a,1}\}$, and let $G'$ be the induced subgraph of $G$ on the vertex set $V=V_1'\cup\cdots\cup V_{r}'.$
Note that induce subgraph $G'$ also satisfies the hypothesis of the theorem with $|V_a'|=n-1<|V_a|$ for all $a$. Hence, by induction hypothesis, $G'$ is Cohen-Macaulay, and hence so is $\dfrac{S}{\langle I(G'), X_{1,1},\ldots,X_{r-1,1}\rangle }$, and $\dim\left(\dfrac{S}{\langle I(G'), X_{1,1},\ldots,X_{r-1,1}\rangle}\right)=n$. By Remark \ref{mulparRem}(v), we know $\dfrac{S}{\langle I(G),X_{1,1},\ldots,X_{r-2,1} \rangle:X_{r-1,1}}$ is Cohen-Macaulay of dimension $n$, and hence so $\dfrac{S}{\langle I(G),X_{1,1},\ldots,X_{r-2,1} \rangle}$. 

Similarly, using the following short exact sequence
{\footnotesize $$
0\rightarrow \dfrac{S(-1)}{\langle I(G),X_{1,1},\ldots,X_{r-3,1} \rangle:X_{r-2,1}}\overset{X_{r-2,1}}\rightarrow\dfrac{S}{\langle I(G),X_{1,1},\ldots,X_{r-3,1} \rangle}\rightarrow \dfrac{S}{\langle I(G),X_{1,1},\ldots,X_{r-2,1} \rangle}\rightarrow 0,
$$}
	
we get $\dfrac{S}{\langle I(G),X_{1,1},\ldots,X_{r-3,1} \rangle}$ is Cohen-Macaulay. By repeating the above process, we get $\dfrac{S}{I(G)}$ is Cohen-Macaulay.	
\end{proof}
Note that if $G$ is Cohen-Macaulay, then $G$ need not be one of the graph described in Theorems \ref{mainTheorem1} and \ref{mainTheorem2}. For example, if $G$ is cycle on $5$ vertices, then $G$ is Cohen-Macaulay. But $G$ does not satisfy the hypothesis of Theorems \ref{mainTheorem1} and \ref{mainTheorem2}.

\begin{remark}{\rm
Let $G$ be as in the above theorem such that for $1\leq a\leq r-1$, the induced graph on vertices $V_a\cup V_r$ is a Herzog-Hibi complete graph. Since ${r-1\choose 2}$ partitions are complete bipartite and $r-1$ partitions are Herzog-Hibi complete, we have the number of edges is $$n^2{r-1\choose 2} +(r-1){n+1\choose 2}={(r-1)n+1\choose 2}.$$ 
}\end{remark}
Since the height of edge ideal of a given graph in the above remark is $(r-1)n$, by \cite[Theorem 4.3.7]{Vi}, we have the following:
 
\begin{corollary}
Let $G$ be given as in the above remark and $I$ be the edge ideal of $G$. Then $S/I$ has a linear resolution.
\end{corollary}


\begin{thebibliography}{99}
\bibitem{ER} Eagon, A., Reiner, V. (1998). Resolutions of Stanley-Reisner Rings and Alexander Duality. \emph{J. pure $\&$ Applied Algebra}. {\bf 130}, 265 - 275.


\bibitem{EHM} Ene, V., Herzog, J., Mohammadi F. (2011). Monomial Ideals and Toric Rings of Hibi Type Arising from a Finite Poset. \emph{European Journal of Combinatorics}. {\bf 32}, 404 - 421. 
 
\bibitem{HH} Herzog, J., Hibi, T. (2005). Distributive lattices, bipartite graphs and Alexander duality. \emph{J. Algebraic Combin.} {\bf 22}, no. 3, 289 -- 302.

\bibitem{HHi} Herzog, J., Hibi, T. (2011).  \emph{Monomial Ideals}, Graduate Texts in Mathematics, Springer-Verlag London Ltd., London. 

\bibitem{ES} Szpilrajn, E. (1930).  Sur l’extension de l’ordre partiel. \emph{Fund. Math.} {\bf 16}, 386 - 389.

\bibitem{Vi}  Villarreal, R. H. (2001). \emph{Monomial Algebras}, Marcel Dekker.

\end{thebibliography}
\end{document}